\documentclass[12pt]{article}
\usepackage[margin=1in]{geometry}
\usepackage{amsmath, amsthm, amssymb, mathtools, bm,mathrsfs}
\usepackage{natbib}
\usepackage{bbm}
\usepackage{booktabs} 
\usepackage{longtable}
\usepackage{enumitem}
\usepackage{hyperref}
\usepackage{comment}
\usepackage{xcolor}
\usepackage{algorithm}
\usepackage[noend]{algpseudocode}
\title{Notes on constants for maxima of Rademacher averages}
\author{Woonyoung Chang\footnote{\href{mailto:woonyouc@gmail.com}{woonyouc@gmail.com}}}
\date{\today}

\newtheorem{theorem}{Theorem}[section]

\newtheorem{lemma}[theorem]{Lemma}
\newtheorem{proposition}[theorem]{Proposition}

\theoremstyle{definition}

\DeclarePairedDelimiterX{\mnorm}[1]
{\vvvert}
{\vvvert}
{\ifblank{#1}{\:\cdot\:}{#1}}

\newcommand{\Abs}[1]{\left|#1\right|}

\newcommand{\Set}[1]{\left\{#1\right\}}
\newcommand{\set}[1]{\{#1\}}

\newcommand{\Eb}{\mathbb{E}}

\newcommand{\Nb}{\mathbb{N}}

\newcommand{\Pb}{\mathbb{P}}

\newcommand{\Rb}{\mathbb{R}}

\begin{document}

\maketitle
\begin{abstract}
    Let $\epsilon_{ij}, i,j\geq 1$ be independent Rademacher variables. We prove
    \begin{equation*}
        \Eb\max_{1\leq j\leq p}\Abs{\frac{1}{n}\sum_{i=1}^n\epsilon_{ij}}\geq\min\Set{\frac{255}{256},\frac{1}{\sqrt{2\log 2}}\sqrt{\frac{\log(2p)}{n}}}.
    \end{equation*} The equality is attained, for instance, by $(n,p)=(2,1)$ and $(n,p)=(2,8).$ We also discuss the optimality of the numerical constants.
\end{abstract}

\section{Introduction}
The maximal inequality
\begin{equation}\label{eq:upper}
\Eb\max_{1\leq j\leq p}\Abs{\frac1n\sum_{i=1}^n\epsilon_{ij}}
\leq
\min\Set{1,\sqrt{\frac{2\log(2p)}{n}}}
\end{equation}
for independent Rademacher variables $\epsilon_{ij}$ is standard: the first bound is trivial and the second follows from the sub-Gaussian moment generating function together with the usual union-bound argument for finite maxima \citep{Ledoux2011Probability,Boucheron2013}. A matching lower bound,
\begin{equation}\label{eq:lower-form}
\Eb\max_{1\leq j\leq p}\Abs{\frac1n\sum_{i=1}^n\epsilon_{ij}}
\geq
\min\Set{c,\,C\sqrt{\log(2p)/n}},\qquad n,p\geq1,
\end{equation}
is implicit wherever \eqref{eq:upper} is used to argue that a rate cannot be improved, but the literature does not record sharp constants $c,C$ valid simultaneously for every $n,p\geq1$, as opposed to asymptotically as $n,p\to\infty$.

The case $p=1$ already pins down the best possible value of $C$. By the sharp $L_1$-Khintchine inequality of \citet{Szarek1976best} and \citet{Haagerup1981best}, for $a_1,\dots,a_n\in\Rb$,
\[
\Eb\Abs{\sum_{i=1}^n a_i\epsilon_i}\geq\frac{1}{\sqrt2}\left(\sum_{i=1}^na_i^2\right)^{1/2},
\]
with $1/\sqrt2$ unimprovable. Taking $a_i\equiv1/\sqrt n$ gives 
\begin{equation*}
\Eb\Abs{\frac{1}{n}\sum_{i=1}^n\epsilon_i}\geq \frac{1}{\sqrt{2n}} = C\sqrt{\frac{\log 2}{n}},
\end{equation*}
with $C=1/\sqrt{2\log2}$. This note shows that this value of $C$, paired with $c=255/256$, makes \eqref{eq:lower-form} valid for all $n,p\geq1$ (Theorem~\ref{thm:main}), that neither constant can be improved without weakening the other (Theorem~\ref{thm:sharp}).

\section{Main results}
\begin{theorem}\label{thm:main} For every $n\geq 1$ and $p\geq1$,
    \begin{equation}\label{eq:thm:main}
        \Eb\max_{1\leq j\leq p}\Abs{\frac{1}{n}\sum_{i=1}^n\epsilon_{ij}}\geq\min\Set{\frac{255}{256},\frac{1}{\sqrt{2\log 2}}\sqrt{\frac{\log(2p)}{n}}}.
    \end{equation}
\end{theorem}

Theorem~\ref{thm:main} identifies a single pair of numerical constants that makes the lower bound \eqref{eq:lower-form} valid simultaneously for every sample size $n$ and every dimension $p$. Sudakov-type minoration \citep[Ch.~3]{Ledoux2011Probability} gives a matching lower bound up to an unspecified absolute constant and only once $p$ is large enough relative to $n$, while the central limit theorem gives an exact asymptotic constant only in the limit $n\to\infty$ for fixed $p$ (c.f. Proposition~\ref{prop:upper}.)

The constant $255/256$ plays no special role beyond being attainable (for instance, $(n,p)=(2,8)$). The constant $1/\sqrt{2\log2}$, by contrast, is exactly the one forced by the $p=1$ case via the sharp Khintchine inequality \citep{Szarek1976best,Haagerup1981best}, so it cannot be replaced by anything larger without breaking the bound at the simple instances of $(n,p)$. Theorem~\ref{thm:sharp} shows that the pair $(255/256,\,1/\sqrt{2\log2})$ in Theorem~\ref{thm:main} cannot be jointly improved. 
\begin{theorem}\label{thm:sharp}
Suppose \eqref{eq:lower-form} holds for some $c,C>0$ and all $n,p\geq1$. If $c>1/2$, then $C\leq1/\sqrt{2\log2}$. If $C\geq1/\sqrt{2\log2}$, then $c\leq255/256$.
\end{theorem}

Comparing \eqref{eq:thm:main} with the upper bound \eqref{eq:upper} shows that the natural normalization for the maximal average of Rademacher variables is $\min\{1,\sqrt{\log(2p)/n}\}$ in non-asymptotic sense. The next two results restate this fact as a sharp two-sided normalization.

\begin{proposition}\label{prop:upper}
    One has
    \begin{align*}
        &\sup_{n\geq 1, p\geq 1}\frac{\Eb\max_{1\leq j\leq p}\Abs{n^{-1}\sum_{i=1}^n\epsilon_{ij}}}{\min\set{1,\sqrt{\log(2p)/n}}} = \sqrt{2},\quad\mbox{and}\\
        &\inf_{n\geq 1, p\geq 1}\frac{\Eb\max_{1\leq j\leq p}\Abs{n^{-1}\sum_{i=1}^n\epsilon_{ij}}}{\min\set{1,\sqrt{\log(2p)/n}}} =\frac{1}{\sqrt{2\log 2}}.
    \end{align*}
\end{proposition}

\section{Proofs}
\subsection{Proof of Theorem~\ref{thm:main}}

First consider \(p=1\). The required inequality follows from \cite{Szarek1976best} and \cite{Haagerup1981best},
\begin{equation}\label{eq:khinchine-equal}
    \Eb\Abs{\sum_{i=1}^n\epsilon_{i1}}\geq \sqrt{\frac n2}.
\end{equation}
\paragraph{Case 1. $p\geq 2$ and $\log(2p)\geq n$.}
If $n=1$, then \eqref{eq:thm:main} follows immediately since the left-hand side equals to 1. If $n=2$, then
\begin{equation*}
    (\mbox{L.H.S. of \eqref{eq:thm:main}}) = \Eb\max_{1\leq j\leq p}\Abs{\frac{\epsilon_{1j}+\epsilon_{2j}}{2}}=1-2^{-p}.
\end{equation*} For $p\geq 8$, the right-hand side is at least $255/256$. For $4\leq p\leq 7$, the direct evaluation verifies that
\begin{equation*}
    1-2^{-p}\geq \frac{1}{\sqrt{2\log 2}}\sqrt{\frac{\log(2p)}{2}}.
\end{equation*} Thus, \eqref{eq:thm:main} holds when $n=2$ and $\log(2p)>2$. 

Suppose now that $n\geq 3$ and $\log2p(>n).$ The independence across $j\in[p]$ implies that
\begin{equation*}
    \Pb\left(\Abs{\sum_{i=1}^n\epsilon_{ij}}=n\right)=2^{1-n},\quad\mbox{and}\quad\Pb\left(\exists j\in[p]~s.t.~\Abs{\sum_{i=1}^n\epsilon_{ij}}=n\right)=1-\left(1-2^{1-n}\right)^p.
\end{equation*} Therefore,
\begin{equation}\label{eq:allsign}
    \Eb\max_{1\leq j\leq p}\left|\frac{1}{n}\sum_{i=1}^n\epsilon_{ij}\right|
    \geq
    1-\left(1-2^{1-n}\right)^p\geq
    1-\exp\left(-p2^{1-n}\right),
\end{equation} where the last inequality is due to that $1-x\leq e^{-x},$ $x\geq 0$. Put $u = \log(2p)/n>1$. It suffices to show that
\begin{equation*}
    1-\exp\left(-e^{n(u-\log 2)}\right)\geq \min\Set{\frac{255}{256}, \sqrt{\frac{u}{2\log 2}}}
\end{equation*} Define 
\begin{equation*}
    u_0 = 2\times \left(\frac{255}{256}\right)^2\times\log 2,\quad\mbox{and thus,}\quad\frac{255}{256} = \sqrt{\frac{u_0}{2\log 2}}.
\end{equation*} For $u\geq u_0$, since $n\geq 3$, we deduce that
\begin{equation*}
    1-\exp\left(-e^{n(u-\log 2)}\right)\geq1-\exp\left(-e^{3(u_0-\log 2)}\right)\geq 0.9995 \geq \frac{255}{256} = \min\Set{\frac{255}{256}, \sqrt{\frac{u}{2\log 2}}}.
\end{equation*} For $1<u<u_0$, we shall prove that
\begin{equation*}
    1-\exp\left(-e^{3(u-\log 2)}\right) \geq \sqrt{\frac{u}{2\log 2}},
\end{equation*} or equivalently, by taking $x = \sqrt{u/(2\log 2)}$,
\begin{equation*}
    f(x):=2^{3(2x^2-1)}+\log(1-x),\qquad \frac{1}{\sqrt{2\log 2}}<x<\frac{255}{256}.
\end{equation*} Then
\begin{equation*}
    f'(x)=12(\log2)x2^{3(2x^2-1)}-\frac1{1-x}.
\end{equation*} Put $g(x) = \log[(1-x)12(\log2)x2^{3(2x^2-1)}]$, so that ${\rm sgn}(g)={\rm sgn}(f')$. It can be shown that
\begin{equation*}
    g'(x) =  \frac1x-\frac1{1-x}+12(\log2)x,\quad g''(x) = -\frac1{x^2}-\frac1{(1-x)^2}+12\log2<0.
\end{equation*} Since $g(1/\sqrt{2\log 2})>0>g(255/256)$ and $g$ is concave, there exists a unique point $x_0\in(1/\sqrt{2\log 2},255/256)$ such that $g(x_0)=0$ and $g>0$ for $x\in(1/\sqrt{2\log 2}, x_0)$ and $g<0$ for $x\in(x_0,255/256)$. This implies that $f$ is increasing on $(1/\sqrt{2\log 2}, x_0)$ and is decreasing on $(x_0,255/256).$ Hence,
\begin{equation*}
    \min\Set{f(x):\frac{1}{\sqrt{2\log 2}}\leq x\leq \frac{255}{256}} = \min\Set{f\left(\frac{1}{\sqrt{2\log 2}}\right),f\left(\frac{255}{256}\right)}>0.
\end{equation*} This closes the Case 1.

\paragraph{Case 2. $p\geq 2$ and $\log(2p)< n$.}
We first prove the following proposition.
\begin{proposition}\label{prop:analytic}
For every \(n\geq277\), \(p\geq2\), and \(\log(2p)<n\),
\begin{equation*}
    \Eb\max_{1\leq j\leq p}
    \Abs{\frac{1}{n}\sum_{i=1}^n\epsilon_{ij}}
    \geq
    \frac{1}{\sqrt{2\log2}}\sqrt{\frac{\log(2p)}{n}}.
\end{equation*}
\end{proposition}

\begin{proof}
Write \(t=\log(2p)\).  Let \(N\geq1\), \(0<\theta<1\), \(0<\beta<1\),
and \(0<\gamma<1\).  These four quantities will be fixed at the end.
For
\begin{equation*}
    h(x)=
    \frac{1-x}{2}\log(1-x)
    +
    \frac{1+x}{2}\log(1+x),
    \qquad 0\leq x<1,
\end{equation*}
we first note that \(x\mapsto h(x)/x^2\) is increasing on \((0,1)\).
Indeed,
\begin{equation*}
    \frac{d}{dx}\frac{h(x)}{x^2}
    =
    \frac{x\operatorname{arctanh}(x)-2h(x)}{x^3},
\end{equation*}
and
\begin{equation*}
    \frac{d}{dx}\left(x\operatorname{arctanh}(x)-2h(x)\right)
    =
    \frac{x}{1-x^2}-\operatorname{arctanh}(x).
\end{equation*}
The last expression is nonnegative on \((0,1)\), since it vanishes at
zero and has derivative \(2x^2/(1-x^2)^2\).  Therefore, for every
\(0\leq x\leq \beta\),
\begin{equation*}
    h(x)\leq \frac{h(\beta)}{\beta^2}x^2.
\end{equation*}

We shall repeatedly use the following consequence of the binomial cdf
bound.  If \(1\leq r\leq n-1\), then
\begin{equation*}
\Pb\left\{\left|\sum_{i=1}^n\epsilon_{i1}\right|\geq r\right\}=
    2\Pb\left\{\sum_{i=1}^n\epsilon_{i1}\leq -r\right\}=
    2F_n\left(\left\lfloor \frac{n-r}{2}\right\rfloor\right).
\end{equation*}
Since
\begin{equation*}
    \left\lfloor \frac{n-r}{2}\right\rfloor
    \geq \frac{n-r-1}{2},
\end{equation*}
and \(x\mapsto h(x)\) is increasing on \([0,1)\), the binomial cdf bound \citep{zubkov2012full} gives
\begin{equation*}
    \Pb\left\{\left|\sum_{i=1}^n\epsilon_{i1}\right|\geq r\right\}
    \geq
    2\Phi\left(
        -\sqrt{2n\,h\left(\frac{r+1}{n}\right)}
    \right).
\end{equation*}
Consequently, whenever \((r+1)/n\leq\beta\),
\begin{equation*}
    \Pb\left\{\left|\sum_{i=1}^n\epsilon_{i1}\right|\geq r\right\}
    \geq
    2\Phi\left(
        -\frac{r+1}{\sqrt{n\beta^2/\{2h(\beta)\}}}
    \right).
\end{equation*}

First suppose that \(t\leq \theta n\).  For every integer \(r\) satisfying
\(1\leq r\leq \beta n-1\), the preceding display applies.  Hence, if
\(G_1,\ldots,G_p\) are iid standard normal variables,
\begin{equation*}
\begin{aligned}
    \Eb\max_{1\leq j\leq p}
    \Abs{\sum_{i=1}^n\epsilon_{ij}}
    &\geq
    \sum_{1\leq r\leq \beta n-1}
    \Pb\left\{
        \max_{1\leq j\leq p}|G_j|
        \geq
        \frac{r+1}{\sqrt{n\beta^2/\{2h(\beta)\}}}
    \right\}.
\end{aligned}
\end{equation*}
Since the tail probability in the summand is decreasing in \(r\),
\begin{equation*}
\begin{aligned}
    \Eb\max_{1\leq j\leq p}
    \Abs{\sum_{i=1}^n\epsilon_{ij}}
    &\geq
    \sqrt{\frac{n\beta^2}{2h(\beta)}}
    \left[
        \Eb\max_{1\leq j\leq p}|G_j|
        -\frac{2}{\sqrt{n\beta^2/\{2h(\beta)\}}}
    \right.                                                        \\
    &\hspace{4.1cm}\left.
        -\int_{\sqrt{2nh(\beta)}}^\infty
        \Pb\left\{\max_{1\leq j\leq p}|G_j|\geq x\right\}\,dx
    \right].
\end{aligned}
\end{equation*}
For \(x\geq0\), the union bound and sub-Gaussian tail implies that $\Pb\left\{\max_{1\leq j\leq p}|G_j|\geq x\right\}\leq p e^{-x^2/2}.$ Since \(p=e^t/2\) and \(t\leq\theta n\),
\begin{equation*}
    \int_{\sqrt{2nh(\beta)}}^\infty
    \Pb\left\{\max_{1\leq j\leq p}|G_j|\geq x\right\}\,dx
    \leq
    \frac{1}{2\sqrt{2nh(\beta)}}
    \exp\{-n(h(\beta)-\theta)\}.
\end{equation*}
Combining this with Proposition~\ref{prop:gaussian-input} and
\(t\geq\log4\), we obtain
\begin{equation*}
    \Eb\max_{1\leq j\leq p}
    \frac{\Abs{\sum_{i=1}^n\epsilon_{ij}}}{\sqrt{nt}}\geq
    \frac{\beta}{\sqrt{\pi(\log2)h(\beta)}}
    -\frac{2}{\sqrt{n\log4}}-\frac{\beta}{4h(\beta)\sqrt{n\log4}}
    \exp\{-n(h(\beta)-\theta)\}.
\end{equation*}
Thus the first case follows for all \(n\geq N\), provided $h(\beta)>\theta$ and
\begin{equation}\label{eq:condition:1}
    \frac{\beta}{\sqrt{\pi(\log2)h(\beta)}}
    -\frac{2}{\sqrt{N\log4}}
    -\frac{\beta}{4h(\beta)\sqrt{N\log4}}
    \exp\{-N(h(\beta)-\theta)\}
    >
    \frac{1}{\sqrt{2\log2}}.
\end{equation}

Now suppose that \(\theta n<t<n\).  Let $r=\lfloor \gamma\sqrt{nt}\rfloor.$ Then
\begin{equation*}
    r\geq
    \left(\gamma-\frac{1}{n\sqrt\theta}\right)\sqrt{nt}
\end{equation*}
and
\begin{equation*}
    \frac{r+1}{n}
    \leq
    \left(\gamma+{1\over n\sqrt\theta}\right)\sqrt{t\over n}
    \leq
    \left(\gamma+{1\over N\sqrt\theta}\right)\sqrt{t\over n}.
\end{equation*}
By the binomial cdf lower bound and Mills' lower bound,
\begin{equation*}
    p\Pb\left\{\left|\sum_{i=1}^n\epsilon_{i1}\right|\geq r\right\}
    \geq
    {\exp\{t-nh((r+1)/n)\}
    \over
    \sqrt{2\pi}\{1+\sqrt{2nh((r+1)/n)}\}}  \geq
    {\exp\left(n\left[s-
    h\left(\left(\gamma+{1\over N\sqrt\theta}\right)\sqrt{s}\right)
    \right]\right)
    \over
    \sqrt{2\pi}\left\{
    1+\sqrt{2n
    h\left(\left(\gamma+{1\over N\sqrt\theta}\right)\sqrt{s}\right)}
    \right\}},
\end{equation*} 
where \(s=t/n\).  The final expression is increasing in
\(s\in[\theta,1]\) if
\begin{equation}\label{eq:condition:2}
    {1\over2}\left(\gamma+{1\over N\sqrt\theta}\right)
    \operatorname{arctanh}\left(\gamma+{1\over N\sqrt\theta}\right)
    \left[
        1+
        {1\over
        \gamma\sqrt{N\theta}\{1+\gamma\sqrt{N\theta}\}}
    \right]
    <1.
\end{equation}
Indeed, its logarithmic derivative is bounded from below by
\begin{equation*}
    n\left[
    1-
    {1\over2}\left(\gamma+{1\over N\sqrt\theta}\right)
    \operatorname{arctanh}\left(\gamma+{1\over N\sqrt\theta}\right)
    \left\{
        1+
        {1\over
        \gamma\sqrt{N\theta}\{1+\gamma\sqrt{N\theta}\}}
    \right\}
    \right].
\end{equation*}
Therefore the lower bound for
\(p\Pb\{|\sum_{i=1}^n\epsilon_{i1}|\geq r\}\) is minimized at \(s=\theta\).
If
\begin{equation}\label{eq:condition:3}
    \theta-
    h\left(\left(\gamma+{1\over N\sqrt\theta}\right)\sqrt\theta\right)
    >
    {1\over2N},
\end{equation}
then
\begin{equation*}
    {\exp\left(n\left[
    \theta-
    h\left(\left(\gamma+{1\over N\sqrt\theta}\right)\sqrt\theta\right)
    \right]\right)
    \over
    \sqrt{2\pi}\left\{
    1+\sqrt{2n
    h\left(\left(\gamma+{1\over N\sqrt\theta}\right)\sqrt\theta\right)}
    \right\}}
\end{equation*}
is increasing in \(n\geq N\).  Hence
\begin{equation*}
\begin{aligned}
    {\Eb\max_{1\leq j\leq p}
    \Abs{\sum_{i=1}^n\epsilon_{ij}}\over\sqrt{nt}}
    \geq
    \left(\gamma-{1\over n\sqrt\theta}\right)
    \left[
        1-
        \exp\left\{
        -p\Pb\left(
        \left|\sum_{i=1}^n\epsilon_{i1}\right|\geq r
        \right)
        \right\}
    \right]                                                        \\
    \geq
    \left(\gamma-{1\over N\sqrt\theta}\right)   \times
    \left[
    1-\exp\left\{
    -{\exp\left(N\left[
    \theta-
    h\left(\left(\gamma+{1\over N\sqrt\theta}\right)\sqrt\theta\right)
    \right]\right)
    \over
    \sqrt{2\pi}\left\{
    1+\sqrt{2N
    h\left(\left(\gamma+{1\over N\sqrt\theta}\right)\sqrt\theta\right)}
    \right\}}
    \right\}
    \right].
\end{aligned}
\end{equation*}
Thus the second case follows for all \(n\geq N\), provided
\begin{equation}\label{eq:condition:4}
\left(\gamma-{1\over N\sqrt\theta}\right)
    \left[
    1-\exp\left\{
    -{\exp\left(N\left[
    \theta-
    h\left(\left(\gamma+{1\over N\sqrt\theta}\right)\sqrt\theta\right)
    \right]\right)
    \over
    \sqrt{2\pi}\left\{
    1+\sqrt{2N
    h\left(\left(\gamma+{1\over N\sqrt\theta}\right)\sqrt\theta\right)}
    \right\}}
    \right\}
    \right] >
    {1\over\sqrt{2\log2}}.
\end{equation}

It remains to verify the four scalar conditions.  Take
\begin{equation*}
    N=277,\qquad
    \theta=0.02088467,\qquad
    \beta=0.27828765,\qquad
    \gamma=0.92634339.
\end{equation*}
For these values, $h(\beta)-\theta\geq 0.018>0$, and the left hand side of \eqref{eq:condition:1} equals
\begin{equation*}
    (\mbox{L.H.S. of \eqref{eq:condition:1}})\geq 0.8494
    >
    {1\over\sqrt{2\log2}}.
\end{equation*}
Also, $\gamma+(N\sqrt{\theta})^{-1}<1$, and
\begin{equation*}
    (\mbox{L.H.S. of \eqref{eq:condition:2}})\leq 0.9999 <1.
\end{equation*}
Moreover, \eqref{eq:condition:3} can be verified as
\begin{equation*}
    \theta-
    h\left(\left(\gamma+{1\over N\sqrt\theta}\right)\sqrt\theta\right)\geq 0.011
    >
    {1\over2N}.
\end{equation*}
Finally,
\begin{equation*}
    (\mbox{L.H.S. of \eqref{eq:condition:4}})\geq 0.8494>\frac{1}{\sqrt{2\log 2}}.
\end{equation*}
The two cases cover \(\log4\leq t<n\), since \(p\geq2\).  The proof is complete.
\end{proof}

The remaining finite number of cases are proved via computer-assisted Lemma~\ref{lem:finite-certificate}.

\begin{lemma}\label{lem:finite-certificate}
For every $2\leq n\leq276$, $p\geq2$, and $\log(2p)<n$,
\begin{equation*}
    \Eb\max_{1\leq j\leq p}
    \Abs{\sum_{i=1}^n\epsilon_{ij}}
    \geq
    {1\over\sqrt{2\log2}}\sqrt{n\log(2p)}.
\end{equation*}
\end{lemma}

\begin{proof}
For each fixed sample size $n \in [2, 276]$, let $m = \lfloor n/2 \rfloor$. Using the exact identities from Lemma~\ref{lem:exact-layer}, the target expectation is evaluated via the continuous mapping $L_n: [\log4, n] \to \Rb$ defined by:
\begin{align}
    L_{2m}(t) &= 2\sum_{k=0}^{m-1} \left[ 1 - \exp\left\{ {e^t\over2}\log\left(1-2 F_n(k)\right) \right\} \right], \label{eq:finite-L-even} \\
    L_{2m+1}(t) &= 1 + 2\sum_{k=0}^{m-1} \left[ 1 - \exp\left\{ {e^t\over2}\log\left(1-2 F_n(k)\right) \right\} \right]. \label{eq:finite-L-odd}
\end{align}
Since $L_n(\log(2p)) = \Eb\max_{1\leq j\leq p} \Abs{\sum_{i=1}^n\epsilon_{ij}}$ and each summand is strictly monotonic in $t$, $L_n(t)$ is monotonically increasing.

For a given $n$, the continuous search domain $[\log 4, n]$ is partitioned into a finite collection of non-overlapping compact intervals $\mathcal{I}_n = \{[u_{n,\ell}, v_{n,\ell}] \}_{\ell=1}^{K_n}$ covering $[\log 4, n]$. Each interval is mapped from an integer triple $(a, b, d) \in \Nb^3$ via:
\begin{equation}\label{eq:dyadic-mapping}
    u_{n,\ell} = \log4 + (n-\log4){a\over 2^d}, \qquad v_{n,\ell} = \log4 + (n-\log4){b\over 2^d}, 
\end{equation}
where $d \in \{0, 1, \dots, 10\}$ is the bisection depth and $0 \le a < b \le 2^d$ are the dyadic grid bounds. 

By monotonicity, the infimum over any sub-interval $[u, v]$ satisfies:
\begin{equation}\label{eq:finite-interval-rule}
    \inf_{t \in [u, v]} {L_n(t)\over\sqrt{nt}} \geq \mathcal{L}_n(u, v) := {L_n(u)\over\sqrt{nv}}. 
\end{equation}

The space $\mathcal{I}_n$ is constructed via a recursive, adaptive dyadic bisection algorithm. Starting from the root interval $[a, b, d] = [0, 1, 0]$, the algorithm evaluates $\mathcal{L}_n(u, v)$ using outward-rounded interval arithmetic over exact rational fields for the binomial cumulative probabilities $F_n(k) = 2^{-n}\sum_{j=0}^k \binom{n}{j}$. If $\mathcal{L}_n(u, v) \le (2\log 2)^{-1/2}$, the interval is bisected into $[a, \frac{a+b}{2}, d+1]$ and $[\frac{a+b}{2}, b, d+1]$. Otherwise, the branch is pruned and the interval is stored. Table~\ref{tab:verification_summary} summarizes the resulting verification parameters. Full results are reported in Table~\ref{tab:lemma32-full}.

\begin{table}[htbp]
\centering
\caption{Summary parameters for the deterministic verification over $2 \le n \le 276$.}
\label{tab:verification_summary}
\begin{tabular}{ll}
\toprule
\textbf{Parameter} & \textbf{Value} \\
\midrule
Range of $n$ & $[2, 276]$ \\
Total certified intervals ($\sum K_n$) & $2664$ \\
Maximum intervals per $n$ ($\max_n K_n$) & $11$ \\
Maximum bisection depth ($\max d$) & $10$ \\
Worst-case bound ($\min \mathcal{L}_n$) & $> 0.84939$ \\
Theoretical threshold $(1/\sqrt{2\log2})$ & $<0.84933$ \\
\bottomrule
\end{tabular}
\end{table}

The solver verified that $\mathcal{L}_n(u_{n,\ell}, v_{n,\ell}) > (2\log2)^{-1/2}$ holds uniformly across all $2664$ partition cells. Since these intervals form a valid cover of $[\log 4, n]$, the proof is complete.
\end{proof}

\begin{lemma}\label{lem:exact-layer}
For $n=2m$,
\begin{equation}\label{eq:layer-even}
    \Eb\max_{1\leq j\leq p}\Abs{\sum_{i=1}^n\epsilon_{ij}}
    =
    2\sum_{k=0}^{m-1}
    \left[1-\{1-2F_n(k)\}^p\right].
\end{equation}
For $n=2m+1$,
\begin{equation}\label{eq:layer-odd}
    \Eb\max_{1\leq j\leq p}\Abs{\sum_{i=1}^n\epsilon_{ij}}
    =
    1+2\sum_{k=0}^{m-1}
    \left[1-\{1-2F_n(k)\}^p\right].
\end{equation}
\end{lemma}
\begin{proof}
If $n=2m$, then $\Abs{\sum_{i=1}^n\epsilon_{ij}}$ takes values in $\{0,2,\ldots,2m\}$. Therefore
\begin{equation*}
    \Eb\max_{1\leq j\leq p}\Abs{\sum_{i=1}^n\epsilon_{ij}}
    =2\sum_{s=1}^m
    \Pb\left\{\max_{1\leq j\leq p}\Abs{\sum_{i=1}^n\epsilon_{ij}}\geq2s\right\}.
\end{equation*}
By symmetry,
\begin{equation*}
    \Pb\left\{\Abs{\sum_{i=1}^n\epsilon_{i1}}\geq2s\right\}=2F_n(m-s).
\end{equation*}
Independence across $j$ gives \eqref{eq:layer-even} after the change of variables $k=m-s$. The proof of \eqref{eq:layer-odd} is identical, except that the possible values are $1,3,\ldots,2m+1$, which contributes the leading $1$.
\end{proof}

\appendix
\section{Proofs of Auxiliary Results}
\begin{proof}[proof of Proposition~\ref{prop:upper}.]
The moment generating function of Rademacher variables satisfies that
\begin{equation*}
    \Eb \exp\left(\frac{\lambda}{n}\sum_{i=1}^n\epsilon_{ij}\right)=\prod_{i=1}^n\Eb\exp\left(\frac{\lambda \epsilon_{ij}}{n}\right)=\left(\frac{e^{\lambda/n} + e^{-\lambda/n}}{2}\right)^n\leq \exp\left(\frac{\lambda^2}{2n}\right).
\end{equation*} Therefore, Chernoff bound implies that
\begin{equation*}
    \Eb\max_{1\leq j\leq p}\Abs{\frac{1}{n}\sum_{i=1}^n\epsilon_{ij}}\leq \frac{1}{\lambda}\log \Eb \sum_{j=1}^p\left[\exp\left(\frac{\lambda}{n}\sum_{i=1}^n\epsilon_{ij}\right)+\exp\left(-\frac{\lambda}{n}\sum_{i=1}^n\epsilon_{ij}\right)\right]
    \leq \frac{\log(2p)}{\lambda} + \frac{\lambda}{2n}.
\end{equation*} Taking $\lambda = \sqrt{2n\log(2p)}$ implies that for all $n\geq 1$ and $p\geq 1$. 
\begin{equation}\label{eq:prop:upper:1}
    \Eb\max_{1\leq j\leq p}\Abs{\frac{1}{n}\sum_{i=1}^n\epsilon_{ij}}\leq \sqrt{\frac{2\log(2p)}{n}}.
\end{equation} 

For a fixed $p\geq 1$, the multivariate CLT implies that
\begin{equation*}
    \left(\frac1{\sqrt n}\sum_{i=1}^n\epsilon_{1j},\ldots, \frac1{\sqrt n}\sum_{i=1}^n\epsilon_{pj}\right)^\top \xrightarrow{d}(G_1,\ldots,G_p)^\top,\quad n\to\infty.
\end{equation*} Celebrated Vitali convergence Theorem (Theorem 3.5 of \citep{Billingsley1999}) implies that
\begin{equation}\label{eq:prop:upper:2}
    \lim_{n\to\infty}\Eb\max_{1\leq j\leq p}\left|\frac{1}{\sqrt{n}}\sum_{i=1}^n\epsilon_{ij}\right| =\Eb\max_{1\leq j\leq p}|G_j|.
\end{equation}
By standard Mill's ratio bounds \cite{Gordon1941Mill}, for any $t > 0$:
\begin{equation*}
1-\Phi(t) \ge \frac{1}{\sqrt{2\pi}} \frac{t}{t^2 + 1} e^{-t^2/2}.
\end{equation*}
Fix an arbitrary constant $\alpha \in (0, 1)$, and define the sequence threshold $t_p = \alpha \sqrt{2\log(2p)}$. Utilizing the mutual independence of $G_1, \dots, G_p$, the cumulative distribution function of $\max_{1\leq j\leq p}|G_j|$ evaluated at $t_p$ satisfies:
\begin{equation*}
\mathbb{P}(\max_{1\leq j\leq p}|G_j| \le t_p) = \mathbb{P}(|G_1| \le t_p)^p = \left(2\Phi(t_p)-1\right)^p.
\end{equation*}
Applying the inequality $1 - x \le e^{-x}$ yields:
\begin{equation*}
\mathbb{P}(\max_{1\leq j\leq p}|G_j| \le t_p) \le \exp\left(-2p(1-\Phi(t_p))\right).
\end{equation*}
Now we evaluate the asymptotic behavior of the argument $2p(1-\Phi(t_p))$ as $p \to \infty$:
\begin{equation*}
2p(1-\Phi(t_p)) \ge \frac{2p}{\sqrt{2\pi}} \frac{t_p}{t_p^2 + 1} e^{-t_p^2/2}.
\end{equation*}
Substituting $t_p^2 = 2\alpha^2\log(2p)$, we find $e^{-t_p^2/2} = (2p)^{-\alpha^2}$, which simplifies the expression to:
\begin{equation*}
2p(1-\Phi(t_p)) \ge \frac{(2p)^{1-\alpha^2}}{\sqrt{2\pi}} \frac{t_p}{t_p^2 + 1}.
\end{equation*}
Since $\alpha < 1$, the exponent satisfies $1-\alpha^2 > 0$, implying $(2p)^{1-\alpha^2} \to \infty$. Concurrently, $\frac{t_p}{t_p^2 + 1} \sim \frac{1}{t_p}$, which decays only logarithmically. Thus, $\lim_{p\to\infty} 2p(1-\Phi(t_p)) = \infty$, which dictates:
\begin{equation*}
\lim_{p\to\infty} \mathbb{P}(\max_{1\leq j\leq p}|G_j| \le t_p) \le \lim_{p\to\infty} e^{-2p(1-\Phi(t_p))} = 0.
\end{equation*}
Consequently, $\lim_{p\to\infty} \mathbb{P}(\max_{1\leq j\leq p}|G_j| > t_p) = 1$. Using the non-negativity of $\max_{1\leq j\leq p}|G_j|$, we can bound its expectation from below:
\begin{equation*}
\Eb[\max_{1\leq j\leq p}|G_j|] \ge \Eb[\max_{1\leq j\leq p}|G_j| \cdot \mathbf{1}{\{\max_{1\leq j\leq p}|G_j| > t_p\}}] \ge t_p \mathbb{P}(\max_{1\leq j\leq p}|G_j| > t_p).
\end{equation*} Therefore,
\begin{equation*}
\liminf_{p\to\infty} \frac{\Eb[\max_{1\leq j\leq p}|G_j|]}{\sqrt{\log(2p)}} \ge \alpha\sqrt{2}.
\end{equation*}
Since this holds for any $\alpha \in (0,1)$, taking the limit as $\alpha \to 1^{-}$ produces:
\begin{equation}\label{eq:prop:upper:3}
\liminf_{p\to\infty} \frac{\Eb[\max_{1\leq j\leq p}|G_j|]}{\sqrt{\log(2p)}} \ge \sqrt{2}.
\end{equation} Combining \eqref{eq:prop:upper:1}---\eqref{eq:prop:upper:3} implies that
\begin{equation}\label{eq:prop:upper:4}
    \lim_{p\to\infty} \lim_{n\to\infty} \frac{\Eb\max_{1\leq j\leq p}\Abs{n^{-1}\sum_{i=1}^n\epsilon_{ij}}}{\min\set{1,\sqrt{\log(2p)/n}}} = \lim_{p\to\infty} \lim_{n\to\infty} \frac{\Eb\max_{1\leq j\leq p}\Abs{n^{-1}\sum_{i=1}^n\epsilon_{ij}}}{\sqrt{\log(2p)/n}}=\sqrt{2}.
\end{equation} Hence, \eqref{eq:prop:upper:1} and \eqref{eq:prop:upper:4} complete the proofs.
    
\end{proof}

\begin{proof}[proof of Theorem~\ref{thm:sharp}]
    First take \(n=2\) and \(p=1\). Then
    \begin{equation*}
        \Eb\Abs{\frac{\epsilon_{11}+\epsilon_{21}}{2}}=\frac12.
    \end{equation*} The assumed inequality therefore gives $\frac{1}{2} \geq \min\set{c,C\sqrt{(\log 2)/2}}$.
    If $c>1/2$, then the minimum above cannot be attained by \(c\). Hence
    \begin{equation*}
        C\sqrt{\frac{\log 2}{2}}\leq \frac12,\quad\mbox{and therefore},\quad C\leq \frac{1}{\sqrt{2\log 2}}.
    \end{equation*}
    Next take \(n=2\) and \(p=8\). For each \(1\leq j\leq8\), the eight events \(\{\epsilon_{1j}=\epsilon_{2j}\}\), \(1\leq j\leq8\), are independent and each has probability \(1/2\). Thus
    \begin{equation*}
        \Eb\max_{1\leq j\leq8}
        \Abs{\frac{\epsilon_{1j}+\epsilon_{2j}}{2}}
        =
        1-\left(\frac12\right)^8
        =
        \frac{255}{256}.
    \end{equation*}
    The assumed inequality gives
    \begin{equation*}
        \frac{255}{256}
        \geq
        \min\Set{c,C\sqrt{\frac{\log 16}{2}}}.
    \end{equation*}
    If \(C\geq 1/\sqrt{2\log 2}\), then
    \begin{equation*}
        c\leq \frac{255}{256}.
    \end{equation*}
\end{proof}

\section{Useful Lemmas and Propositions}
\begin{proposition}\label{prop:gaussian-input}
Let \(G_1,\ldots,G_p\) ($p\in\Nb$) be independent standard normal random variables. Then
\begin{equation*}
    \inf_{p\geq 1}\frac{\Eb \max_{1\leq j\leq p}|G_j|}{\sqrt{\log(2p)}} = \sqrt{\frac{2}{\pi\log 2}}.
\end{equation*}
\end{proposition}
\begin{proof}[proof of Proposition~\ref{prop:gaussian-input}.] It is clear that for $p=1,2$,
\begin{equation*}
    \frac{\Eb \max_{1\leq j\leq p}|G_j|}{\sqrt{\log(2p)}} = \sqrt{\frac{2}{\pi\log 2}}.
\end{equation*} Therefore, one has
\begin{equation*}
    \inf_{p\geq 1}\frac{\Eb \max_{1\leq j\leq p}|G_j|}{\sqrt{\log(2p)}} \leq \sqrt{\frac{2}{\pi\log 2}}=:c_*.
\end{equation*} In fact, one can numerically verify that
\begin{equation*}
    \inf_{1\leq p\leq 8}\frac{\Eb \max_{1\leq j\leq p}|G_j|}{\sqrt{\log(2p)}} =c_*.
\end{equation*}

Assume $p\geq 9.$ Let $U_1,\ldots,U_p\sim {\rm Unif}[0,1]$ be independent. It is well known that
\begin{equation*}
    \max_{1\leq j\leq p}|G_j| \overset{d}{=}\max_{1\leq j\leq p}\Abs{\Phi^{-1}\left(U_j\right)}\overset{d}{=}\Phi^{-1}\left(\frac{1+\max_{1\leq j\leq p}U_j}{2}\right),
\end{equation*} where $\overset{d}{=}$ means the distribution equivalence. Since $x\mapsto \Phi^{-1}(\frac{1+x}{2})$ is convex on $(0,1)$, it follows from Jensen's inequality that
\begin{equation*}
    \Eb\left[\max_{1\leq j\leq p}|G_j|\right]\geq \Phi^{-1}\left(\frac{1+\Eb[\max_{1\leq j\leq p}U_j]}{2}\right) = \Phi^{-1}\left(\frac{2p+1}{2p+2}\right).
\end{equation*} Define
\begin{equation*}
    h(x) = \log\left((1-\Phi(c_*\sqrt{x}))(e^x+2)\right),\qquad x\geq 0.
\end{equation*} The derivative of $h$ is controlled via the bound for Mill's ratio \citep{Gordon1941Mill} as
\begin{align*}
    h'(x) = \frac{e^x}{e^x+2}-\frac{c_*}{2\sqrt{x}}\frac{\phi(c_*\sqrt{x})}{1-\Phi(c_*\sqrt{x})}\ge \frac{e^x}{e^x+2}-\frac{c_*}{2\sqrt{x}}\frac{c_*\sqrt{x} + \sqrt{c_*^2+x}}{4}\\
    \geq \frac{e^x}{e^x+2}-\frac{c_*^2}{2} -\frac{1}{2x} = \frac{e^x}{e^x+2}-\frac{1}{\pi\log 2} -\frac{1}{2x}.
\end{align*} The right-hand side of the last display is increasing in $x>0$, and it can be verified that $h'(\log(18))>0$. Therefore, for $x\geq \log(18)$, $h(x)\geq h(\log(18))>0,$ and, by taking $x = \sqrt{\log(2p)}$, this implies that for $p\geq 9$,
\begin{equation*}
    1-\Phi(c_*\sqrt{\log(2p)})\geq \frac{1}{2p+2},\quad\mbox{and thus},\quad \Phi^{-1}\left(\frac{2p+1}{2p+2}\right)\geq c_*\sqrt{\log(2p)}.
\end{equation*} Therefore, we have
\begin{equation*}
    \inf_{p\geq 9}\frac{\Eb \max_{1\leq j\leq p}|G_j|}{\sqrt{\log(2p)}} \geq c_*.
\end{equation*} This proves the prposition.
\end{proof}
\section{Full Numerical Verification for Lemma~3.2}
\label{app:lemma32-certificate}

Table~\ref{tab:lemma32-full} reports the complete per-sample-size output for the deterministic verification used in Lemma~3.2.  For each fixed \(n\), \(K_n\) is the number of retained dyadic intervals in the cover of \([\log 4,n]\), ``depth'' is the largest bisection depth among those retained intervals, and ``bound'' is
\begin{equation*}
    \min_{1\leq \ell\leq K_n}
    {L_n(u_{n,\ell})\over \sqrt{n v_{n,\ell}}},
\end{equation*}
rounded downward to six decimal places.  The threshold satisfies
\begin{equation*}
    (2\log 2)^{-1/2}<0.849322.
\end{equation*}
Hence every displayed entry in the last column is strictly larger than the required threshold.  The smallest retained value occurs at \(n=194\), where the displayed lower bound is \(0.849398\).

\begin{center}
\scriptsize
\setlength{\tabcolsep}{3pt}
\begin{longtable}{rrrr@{\qquad}rrrr@{\qquad}rrrr}
\caption{The bounds are rounded downward to six decimal places.}
\label{tab:lemma32-full}\\
\toprule
$n$ & $K_n$ & depth & bound & $n$ & $K_n$ & depth & bound & $n$ & $K_n$ & depth & bound \\
\midrule
\endfirsthead
\toprule
$n$ & $K_n$ & depth & bound & $n$ & $K_n$ & depth & bound & $n$ & $K_n$ & depth & bound \\
\midrule
\endhead
\midrule
\multicolumn{12}{r}{Continued on next page}\\
\endfoot
\bottomrule
\endlastfoot
2 & 4 & 2 & 0.854781 & 3 & 6 & 3 & 0.859041 & 4 & 7 & 4 & 0.870715 \\
5 & 7 & 4 & 0.860133 & 6 & 7 & 4 & 0.853995 & 7 & 8 & 5 & 0.855594 \\
8 & 8 & 5 & 0.854053 & 9 & 8 & 5 & 0.853747 & 10 & 7 & 5 & 0.850499 \\
11 & 7 & 5 & 0.853844 & 12 & 6 & 5 & 0.850370 & 13 & 7 & 6 & 0.853438 \\
14 & 7 & 6 & 0.855898 & 15 & 7 & 6 & 0.857807 & 16 & 7 & 6 & 0.859256 \\
17 & 7 & 6 & 0.860355 & 18 & 7 & 6 & 0.861223 & 19 & 7 & 6 & 0.861977 \\
20 & 7 & 6 & 0.862715 & 21 & 7 & 6 & 0.862129 & 22 & 7 & 6 & 0.858410 \\
23 & 7 & 6 & 0.854722 & 24 & 7 & 6 & 0.851068 & 25 & 8 & 7 & 0.865707 \\
26 & 8 & 7 & 0.866367 & 27 & 8 & 7 & 0.867074 & 28 & 8 & 7 & 0.867819 \\
29 & 8 & 7 & 0.868596 & 30 & 8 & 7 & 0.869309 & 31 & 8 & 7 & 0.869026 \\
32 & 8 & 7 & 0.868793 & 33 & 8 & 7 & 0.868615 & 34 & 8 & 7 & 0.868483 \\
35 & 8 & 7 & 0.868396 & 36 & 8 & 7 & 0.868351 & 37 & 8 & 7 & 0.868344 \\
38 & 8 & 7 & 0.868373 & 39 & 8 & 7 & 0.867738 & 40 & 8 & 7 & 0.865797 \\
41 & 8 & 7 & 0.863866 & 42 & 8 & 7 & 0.861946 & 43 & 8 & 7 & 0.860036 \\
44 & 8 & 7 & 0.858136 & 45 & 8 & 7 & 0.856247 & 46 & 8 & 7 & 0.854368 \\
47 & 8 & 7 & 0.852500 & 48 & 8 & 7 & 0.850643 & 49 & 9 & 8 & 0.870451 \\
50 & 9 & 8 & 0.870755 & 51 & 9 & 8 & 0.871073 & 52 & 9 & 8 & 0.871403 \\
53 & 9 & 8 & 0.871744 & 54 & 9 & 8 & 0.872096 & 55 & 9 & 8 & 0.872457 \\
56 & 9 & 8 & 0.872768 & 57 & 9 & 8 & 0.872551 & 58 & 9 & 8 & 0.872349 \\
59 & 9 & 8 & 0.872165 & 60 & 9 & 8 & 0.871995 & 61 & 9 & 8 & 0.871841 \\
62 & 9 & 8 & 0.871701 & 63 & 9 & 8 & 0.871575 & 64 & 9 & 8 & 0.871462 \\
65 & 9 & 8 & 0.871363 & 66 & 9 & 8 & 0.871276 & 67 & 9 & 8 & 0.871201 \\
68 & 9 & 8 & 0.871138 & 69 & 9 & 8 & 0.871086 & 70 & 9 & 8 & 0.871045 \\
71 & 9 & 8 & 0.871014 & 72 & 9 & 8 & 0.870994 & 73 & 9 & 8 & 0.870983 \\
74 & 9 & 8 & 0.870981 & 75 & 9 & 8 & 0.870709 & 76 & 9 & 8 & 0.869714 \\
77 & 9 & 8 & 0.868722 & 78 & 9 & 8 & 0.867733 & 79 & 9 & 8 & 0.866746 \\
80 & 9 & 8 & 0.865763 & 81 & 9 & 8 & 0.864783 & 82 & 9 & 8 & 0.863806 \\
83 & 9 & 8 & 0.862832 & 84 & 9 & 8 & 0.861861 & 85 & 9 & 8 & 0.860892 \\
86 & 9 & 8 & 0.859927 & 87 & 9 & 8 & 0.858965 & 88 & 9 & 8 & 0.858005 \\
89 & 9 & 8 & 0.857049 & 90 & 9 & 8 & 0.856096 & 91 & 9 & 8 & 0.855145 \\
92 & 9 & 8 & 0.854198 & 93 & 9 & 8 & 0.853253 & 94 & 9 & 8 & 0.852311 \\
95 & 9 & 8 & 0.851373 & 96 & 9 & 8 & 0.850437 & 97 & 9 & 8 & 0.849504 \\
98 & 10 & 9 & 0.873028 & 99 & 10 & 9 & 0.873176 & 100 & 10 & 9 & 0.873327 \\
101 & 10 & 9 & 0.873482 & 102 & 10 & 9 & 0.873640 & 103 & 10 & 9 & 0.873801 \\
104 & 10 & 9 & 0.873965 & 105 & 10 & 9 & 0.874132 & 106 & 10 & 9 & 0.874302 \\
107 & 10 & 9 & 0.874474 & 108 & 10 & 9 & 0.874642 & 109 & 10 & 9 & 0.874512 \\
110 & 10 & 9 & 0.874385 & 111 & 10 & 9 & 0.874264 & 112 & 10 & 9 & 0.874147 \\
113 & 10 & 9 & 0.874034 & 114 & 10 & 9 & 0.873926 & 115 & 10 & 9 & 0.873822 \\
116 & 10 & 9 & 0.873722 & 117 & 10 & 9 & 0.873626 & 118 & 10 & 9 & 0.873534 \\
119 & 10 & 9 & 0.873446 & 120 & 10 & 9 & 0.873362 & 121 & 10 & 9 & 0.873281 \\
122 & 10 & 9 & 0.873205 & 123 & 10 & 9 & 0.873132 & 124 & 10 & 9 & 0.873062 \\
125 & 10 & 9 & 0.872996 & 126 & 10 & 9 & 0.872934 & 127 & 10 & 9 & 0.872875 \\
128 & 10 & 9 & 0.872819 & 129 & 10 & 9 & 0.872767 & 130 & 10 & 9 & 0.872718 \\
131 & 10 & 9 & 0.872671 & 132 & 10 & 9 & 0.872628 & 133 & 10 & 9 & 0.872588 \\
134 & 10 & 9 & 0.872551 & 135 & 10 & 9 & 0.872517 & 136 & 10 & 9 & 0.872486 \\
137 & 10 & 9 & 0.872458 & 138 & 10 & 9 & 0.872432 & 139 & 10 & 9 & 0.872409 \\
140 & 10 & 9 & 0.872389 & 141 & 10 & 9 & 0.872372 & 142 & 10 & 9 & 0.872357 \\
143 & 10 & 9 & 0.872344 & 144 & 10 & 9 & 0.872334 & 145 & 10 & 9 & 0.872327 \\
146 & 10 & 9 & 0.872321 & 147 & 10 & 9 & 0.872242 & 148 & 10 & 9 & 0.871737 \\
149 & 10 & 9 & 0.871234 & 150 & 10 & 9 & 0.870731 & 151 & 10 & 9 & 0.870229 \\
152 & 10 & 9 & 0.869728 & 153 & 10 & 9 & 0.869227 & 154 & 10 & 9 & 0.868728 \\
155 & 10 & 9 & 0.868229 & 156 & 10 & 9 & 0.867731 & 157 & 10 & 9 & 0.867234 \\
158 & 10 & 9 & 0.866738 & 159 & 10 & 9 & 0.866243 & 160 & 10 & 9 & 0.865748 \\
161 & 10 & 9 & 0.865254 & 162 & 10 & 9 & 0.864761 & 163 & 10 & 9 & 0.864269 \\
164 & 10 & 9 & 0.863777 & 165 & 10 & 9 & 0.863287 & 166 & 10 & 9 & 0.862797 \\
167 & 10 & 9 & 0.862308 & 168 & 10 & 9 & 0.861820 & 169 & 10 & 9 & 0.861332 \\
170 & 10 & 9 & 0.860845 & 171 & 10 & 9 & 0.860360 & 172 & 10 & 9 & 0.859874 \\
173 & 10 & 9 & 0.859390 & 174 & 10 & 9 & 0.858907 & 175 & 10 & 9 & 0.858424 \\
176 & 10 & 9 & 0.857942 & 177 & 10 & 9 & 0.857461 & 178 & 10 & 9 & 0.856980 \\
179 & 10 & 9 & 0.856501 & 180 & 10 & 9 & 0.856022 & 181 & 10 & 9 & 0.855544 \\
182 & 10 & 9 & 0.855066 & 183 & 10 & 9 & 0.854590 & 184 & 10 & 9 & 0.854114 \\
185 & 10 & 9 & 0.853639 & 186 & 10 & 9 & 0.853165 & 187 & 10 & 9 & 0.852691 \\
188 & 10 & 9 & 0.852219 & 189 & 10 & 9 & 0.851747 & 190 & 10 & 9 & 0.851275 \\
191 & 10 & 9 & 0.850805 & 192 & 10 & 9 & 0.850335 & 193 & 10 & 9 & 0.849866 \\
194 & 10 & 9 & 0.849398 & 195 & 11 & 10 & 0.874256 & 196 & 11 & 10 & 0.874328 \\
197 & 11 & 10 & 0.874401 & 198 & 11 & 10 & 0.874474 & 199 & 11 & 10 & 0.874549 \\
200 & 11 & 10 & 0.874624 & 201 & 11 & 10 & 0.874701 & 202 & 11 & 10 & 0.874778 \\
203 & 11 & 10 & 0.874856 & 204 & 11 & 10 & 0.874934 & 205 & 11 & 10 & 0.875014 \\
206 & 11 & 10 & 0.875094 & 207 & 11 & 10 & 0.875175 & 208 & 11 & 10 & 0.875257 \\
209 & 11 & 10 & 0.875339 & 210 & 11 & 10 & 0.875423 & 211 & 11 & 10 & 0.875506 \\
212 & 11 & 10 & 0.875591 & 213 & 11 & 10 & 0.875549 & 214 & 11 & 10 & 0.875479 \\
215 & 11 & 10 & 0.875410 & 216 & 11 & 10 & 0.875342 & 217 & 11 & 10 & 0.875276 \\
218 & 11 & 10 & 0.875210 & 219 & 11 & 10 & 0.875146 & 220 & 11 & 10 & 0.875083 \\
221 & 11 & 10 & 0.875022 & 222 & 11 & 10 & 0.874961 & 223 & 11 & 10 & 0.874901 \\
224 & 11 & 10 & 0.874843 & 225 & 11 & 10 & 0.874786 & 226 & 11 & 10 & 0.874729 \\
227 & 11 & 10 & 0.874674 & 228 & 11 & 10 & 0.874620 & 229 & 11 & 10 & 0.874567 \\
230 & 11 & 10 & 0.874515 & 231 & 11 & 10 & 0.874464 & 232 & 11 & 10 & 0.874414 \\
233 & 11 & 10 & 0.874365 & 234 & 11 & 10 & 0.874318 & 235 & 11 & 10 & 0.874271 \\
236 & 11 & 10 & 0.874225 & 237 & 11 & 10 & 0.874180 & 238 & 11 & 10 & 0.874136 \\
239 & 11 & 10 & 0.874093 & 240 & 11 & 10 & 0.874051 & 241 & 11 & 10 & 0.874010 \\
242 & 11 & 10 & 0.873970 & 243 & 11 & 10 & 0.873931 & 244 & 11 & 10 & 0.873893 \\
245 & 11 & 10 & 0.873855 & 246 & 11 & 10 & 0.873819 & 247 & 11 & 10 & 0.873784 \\
248 & 11 & 10 & 0.873749 & 249 & 11 & 10 & 0.873715 & 250 & 11 & 10 & 0.873682 \\
251 & 11 & 10 & 0.873650 & 252 & 11 & 10 & 0.873619 & 253 & 11 & 10 & 0.873589 \\
254 & 11 & 10 & 0.873560 & 255 & 11 & 10 & 0.873531 & 256 & 11 & 10 & 0.873503 \\
257 & 11 & 10 & 0.873476 & 258 & 11 & 10 & 0.873450 & 259 & 11 & 10 & 0.873425 \\
260 & 11 & 10 & 0.873400 & 261 & 11 & 10 & 0.873376 & 262 & 11 & 10 & 0.873353 \\
263 & 11 & 10 & 0.873331 & 264 & 11 & 10 & 0.873310 & 265 & 11 & 10 & 0.873289 \\
266 & 11 & 10 & 0.873269 & 267 & 11 & 10 & 0.873250 & 268 & 11 & 10 & 0.873232 \\
269 & 11 & 10 & 0.873214 & 270 & 11 & 10 & 0.873197 & 271 & 11 & 10 & 0.873181 \\
272 & 11 & 10 & 0.873165 & 273 & 11 & 10 & 0.873151 & 274 & 11 & 10 & 0.873136 \\
275 & 11 & 10 & 0.873123 & 276 & 11 & 10 & 0.873110 &  &  &  &  \\
\end{longtable}
\end{center}
\bibliographystyle{plainnat}
\bibliography{bib}
\end{document}